\documentclass[11pt]{article}
\usepackage{epsfig}
\usepackage{graphicx}
\usepackage{amssymb,amsmath}
\usepackage{caption}
\usepackage{cancel}
\usepackage{tikz}
\newtheorem{theorem}{{\bf Theorem}}[section]
\newtheorem{thm}{{\bf Theorem}}[section]

\newtheorem{lemma}[thm]{\bf Lemma}

\newtheorem{corollary}[thm]{\bf Corollary}

\newtheorem{definition}[thm]{\bf Definition}
\newtheorem{preproof}{{\bf Proof.}}

\newenvironment{proof}[1]{\begin{preproof}{\rm
               #1}\hfill{$\rule{2mm}{2mm}$}}{\end{preproof}}
\begin{document}
\title{\Large {\bf Degree sequence of the generalized Sierpi\'{n}ski graph}}

\author{
{ Ali Behtoei\thanks{Corresponding author, a.behtoei@sci.ikiu.ac.ir}}, { Fatemeh Attarzadeh\thanks{prs.attarzadeh@gmail.com}}, { Mahsa Khatibi\thanks{m.khatibi@edu.ikiu.ac.ir}} \\
[1mm]
{\it \small Department of Mathematics, Faculty of Science, Imam Khomeini International University,} \\
 { \it \small Qazvin, Iran, PO Box: 34148 - 96818.
}  }
\date{}
\maketitle
\vspace*{-5mm}
\begin{abstract}
The degree sequence of (ordinary) Sierpi\'nski graphs and Hanoi graphs  are determined in the literature. Also, 
In
[$\!$ J.A. Rodriguez-Vel\'azquez, E.D. Rodriguez-Bazan, A. Estrada-Moreno, On generalized Sierpi\'{n}ski graphs,  Discuss. Math. Graph T., 37 (3), 2017, 547-560.] the authors
determine the number of leaves (vertices of degree one)  of the generalized Sierpi\'nski $S(T, t)$ of any tree $T$ in terms of $t$, $|V(T)|$ and  the number of leaves of the base graph $T$. In this paper, among some other results, we generalize these results. More precisely, for every simple graph $G$ of order $n$, we completely determine the degree sequence of the generalized Sierpi\'nski graph $S(G,t)$ of $G$ in terms of $n$, $t$ and the degree sequence of $G$. Also, we determine the exact value of the general first  Zagreb index of $S(G,t)$ in terms of the same parameters of $G$. 
\end{abstract}



\section{Introduction}

All graphs considered in this paper are assumed to be simple and finite. 
Throughout the paper, $G=(V,E)$ will denote a  graph of order $n=|V|$ with the vertex set $V$ and the edge set $E$.
The degree of  a vertex $v$ of $G$ is denoted by $\deg_G(v)$ which is the size of the set of  its neighbourhood $N_G(v)$. 
Decomposition into special substructures inheriting significant properties is an important method for the investigation of some mathematical structures,  
especially when the considered structures have self-similarity properties. 
In these cases we typically only need to study the substructures and the way that they are related together. 
For example polymer networks can be modeled by generalized Sierpi\'nski graphs, see \cite{GeneralRandic}.
Sierpi\'nski and Sierpi\'nski-type Graphs are studied in fractal theory \cite{Teplyaev} and appear naturally in diverse areas of mathematics and in several scientific fields.
This family of graphs were studied for the first time in \cite{FirstTime1} and  \cite{Pisanski}, independently, and
 constitutes an extensively studied class of graphs of fractal nature with applications in computer science (as a
model for interconnection networks which are known as WK-recursive networks), topology and mathematics of the Tower of Hanoi, see \cite{Della} and \cite{Romik} for more details.
One of the most important families of such graphs is formed by the Sierpi\'nski gasket graphs  introduced  by Scorer, Grundy and Smith in \cite{gasket-graphs}
which play an important role in psychology, dynamic systems and probability, see  \cite{EuropeanJ-Sierpinski gasket}, \cite{ArsComb} and \cite{Klix}.
Sierpi\'nski, Sierpi\'nski-type and generalized Sierpi\'nski graphs have many interesting properties and were studied extensively in  literature.

\begin{definition} \cite{EuroComb2011}
Let $G=(V, E)$ be a graph of order $n\geq 2$,  $t$ be a positive integer and
denote the set of words of length $t$ on the alphabet $V$  by $V^t$. The letters of a
word ${\bf u} \in V^t$ (of length $t$) are denoted by $u_1u_2...u_t$.
The generalized Sierpi\'nski graph of $G$ of dimension $t$, denoted by $S(G,t)$, is the graph with vertex set $V^t$ and $\{ {\bf u} , {\bf v} \}$ is an edge in it if and only if there exists $i \in\lbrace1,...,t\rbrace$ such that:\\
 (i)~ $u_j = v_j ~~~~ if  ~~~~  j < i$, \\
 (ii) ~$u_i \neq v_i ~~~~~and ~~~\{u_i,v_i\} \in E(G)$, \\
 (iii) ~$u_j = v_i ~~~and~~~ v_j = u_i~~ if~~ j > i$.
 \end{definition}

 For example $S(C_4,3)$ is depicted in Figure  \ref{fig:pic93} 
 in which $C_4$ is assumed to be a cycle with the vertex set $\{1,2,3,4\}$ and the edge set $\{\{1,2\},\{2,3\},\{3,4\},\{4,1\}\}$.
Note that $S(G,1)$ is (isomorphic to) the base graph $G$ and  $S(G,2)$ can be constructed by copying $n$ times $S(G,1)$ and adding an edge between the $i$-th vertex of the $j$-th copy and the $j$-th vertex of the $i$-th copy of $S(G,1)$  whenever $\{i,j\}$ is an edge in $G$.  
In fact $S(G,t)$ is a fractal-like graph that uses $G$ as a building block.
When  $G$ is the complete graph $K_n$, the (ordinary) Sierpi\'nski graph $S^t_n=S(K_n,t)$ is obtained. 
Klav$\check{z}$ar et al.  introduced  the graph $S(K_n, t)$ for the first time and they show that
 $S(K_3, t)$ is isomorphic to the graph of the Tower of Hanoi $H^t_3$, see \cite{FirstTime1} and \cite{Perfect-Codes}.
It is well known that  $S^t_n$ contains $n$  (extreme) vertices of degree $n-1$ and all the other vertices are of degree $n$, see \cite{HinzSurvey}. 
Hence, $S^t_n$ is  almost regular. 
The degree sequence of  the Hanoi graph  $H^t_n$  (the state graph of the Tower of Hanoi game with $n$ pegs and $t$ discs) is  more complex and is completely determined in \cite{Book2013}. 
In this paper, we want to  determine the degree sequence of the generalized Sierpi\'nski graph $S(G,t)$ of an arbitrary graph $G$.
Sierpi\'nski graphs are studied  from numerous points of view.
In \cite{FirstTime1} and \cite{ShortestPaths}  shortest paths in Sierpi\'nski graphs are studied.
In \cite{Automata}  an algorithm is proposed which makes use of three automata to determine all shortest paths in Sierpi\'nski graphs.
Metric properties of Sierpi\'nski graphs is investigated in \cite{Hinz} and \cite{Parisse}.
For connections between the Sierpi\'nski graphs  and Stern’s diatomic sequence see \cite{EuropeanJ2005}.
Identifying codes, locating-dominating codes and total-dominating codes in Sierpi\'nski graphs are studied in \cite{Gravier2013}.
Sierpi\'nski graphs contain (essentially) unique 1-perfect codes \cite{Perfect-Codes}.
Covering codes in Sierpi\'nski graphs is studied in \cite{Beaudou} and  equitable $L(2,1)$-labelings of them is considered in \cite{FuXie}.
In \cite{Sandi-European} the canonical isometric representation of Sierpi\'nski graphs is  explicitly described.
The crossing number of Sierpi\'nski graphs is studied in \cite{Mohar}, giving the first infinite families of graphs of fractal nature for which the crossing number is determined (up to the crossing number of complete graphs).
Colorings and the chromatic number of these graphs is studied in \cite{H-P-Coloring} and the hub number of them is determined in \cite{hub-Number}.
\begin{figure}[ht]
\centering
\includegraphics[scale=.5]{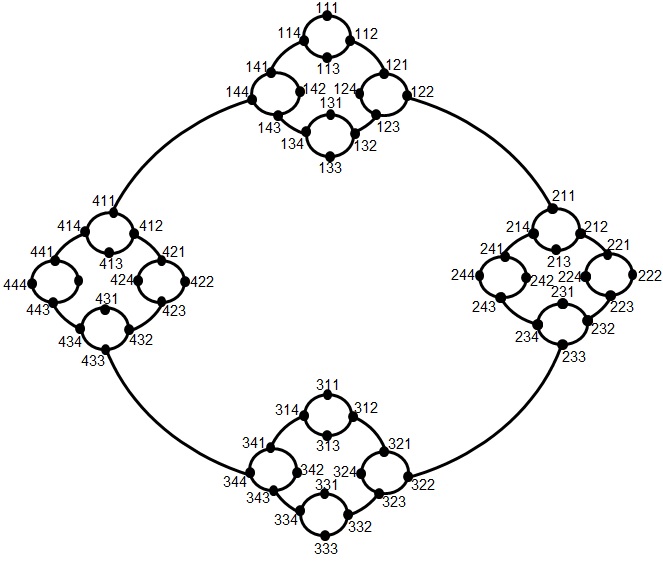}
\caption{\label{fig:pic93}
The generalized Sierpi\'nski graph $S(C_4,3)$.
}
\end{figure}
Also, many papers studied the structure of the generalized Sierpi\'nski graphs. 
In \cite{EuroComb2011}  some interesting results about the generalized Sierpi\'nski graphs (concerning their  automorphism groups, perfect codes 
and  distinguishing numbers) are obtained. 
The total chromatic number for some families of these graphs is determined by Geetha and Somasundaram in \cite{Australasian2015}. 
More precisely,  the authors prove the tight bound of the Behzad and Vizing conjecture on total coloring for the generalized Sierpi\'nski graphs
of cycle graphs and hypercube graphs. They provide a total coloring for the
WK-recursive topology, which also gives the tight bound.
In \cite{Distance2016} the distance between vertices of  $S(G, t)$ is expressed in terms of the distance between vertices of the base graph $G$. 
In addition, the authors give an explicit formula for the diameter and the radius of $S(G, t)$ when the base graph $G$ is a tree.
In \cite{2017Alberto} their  independence number, chromatic number, vertex cover number, clique number and domination number are investigated in terms of the similar parameters of the base graph $G$.
The strong metric dimension of these graphs is studied in \cite{Estaji}.
Metric properties of generalized Sierpi\'nski graphs over stars are considered in \cite{Alizadeh}.
The Roman domination number of $S (G, t )$ is investigated in \cite{Ramezani}.
An explicit formula for the number of connected components of $S(G, t)$ is given in \cite{ConnectivityZ} and it is proved that the (edge-)connectivity of $S(G, t)$ equals the (edge-) connectivity of $G$. Also, It is demonstrated that $S(G, t)$ contains a perfect matching if and only if $G$ contains a perfect matching. Moreover, Hamiltonicity of these graphs is also discussed there.
It is shown in \cite{2017Alberto}  that for any tree $T$ of order $n\geq 2$ and any positive integer $t$, $S(T, t)$ is a tree and 
the number of leaves of $S(T, t)$ is  equal to 
$${\varepsilon(T)~(n^t-2n^{t-1}+1) \over n-1}$$ 
in which $\varepsilon(T)$ is the number of leaves of $T$. We generalize this result in Theorem \ref{DegreeSequence}.
For more results in these subject and related subjects, see \cite{EuroComb2011}, \cite{HinzSurvey}, \cite{Book2013}, \cite{Farahani}, \cite{2017Alberto} and \cite{Randic2015}.

\section{Main results}

First, we determine the neighbourhood of a vertex in $S(G,t)$ and hence, its degree in $S(G,t)$. 
Then, for each $0\leq k \leq \Delta(S(G,t))$, we determine the number of vertices of degree $k$ in $S(G,t)$ which leads to the degree sequence of the generalized Sierpi\'nski graph $S(G,t)$. 
Finally, we show that  the general first  Zagreb index  of $S(G,t)$ can be expressed as a linear combination of the general first  Zagreb indices  of the base graph $G$.

\begin{lemma}  \label{VertexDegree}
Let $G=(V,E)$ be a simple graph  and $t\geq 1$ be an integer. Then,
for each vertex ${\bf x}=x_1x_2...x_t$ in the generalized Sierpi\'nski graph $S(G,t)$ we have
\begin{eqnarray*}
\deg_{S(G,t)} (x_1x_2...x_t)= 
\left\{ \begin{array}{ll}  
1+\deg_G(x_t) &  ~~  x_1x_2...x_t\neq x_tx_t...x_t ~and~ ~ x_i \in N_G(x_t)  ~ for \\ & ~ ~ i={\max\{j:~1\leq j\leq t-1,~ x_j\neq x_t\}}, \\ \\  
\deg_G(x_t) &  ~~ otherwise.   
\end{array} \right.
\end{eqnarray*}
\end{lemma}
\begin{proof}{
Obviously the result follows for $t=1$ because $x_1...x_t=x_t...x_t=x_t$ and  $\deg_{S(G,1)}(x_1)=\deg_G(x_1)$. Hence, we assume that $t\geq 2$.
By  the adjacency rule in $S(G,t)$, it is straightforward to see that each vertex ${\bf x'}=x'_1x'_2...x'_t$  in $S(G,t)$ with $N_G(x'_t)=\emptyset$  is an isolated vertex in $S(G,t)$. Thus, If  $\deg_G(x_t)=0$, then ${\bf x}=x_1x_2...x_t$ is an isolated vertex in $S(G,t)$ and the result directly follows.
Assume that $\deg_G(x_t)=d\geq 1$ and $N_G(x_t)=\{y_1,y_2,...,y_d\}$.  
For each $i\in \{1,2,...,d\}$, it is easy to check that two vertices ${\bf x}=x_1x_2...x_{t-1}x_t$ and ${\bf y_i}=x_1x_2...x_{t-1}y_i$ are adjacent in $S(G,t)$. 
Hence,  $\deg_{S(G,t)}({\bf x}) \geq d= \deg_G(x_t)$. 
Now let ${\bf z}=z_1z_2...z_t $ be a neighbour of ${\bf x}$ in $S(G,t)$.
Thus,  by the adjacency rule in $S(G,t)$, there exists $i\in \{1,2,...,t\}$ such that $x_j=z_j$ for each $j<i$, $\{x_i,z_i\}\in E(G)$ and ($x_\ell =z_i$ and $z_\ell=x_i$) for each $\ell >i$. 
If $i=t$, then ${\bf z}=x_1x_2...x_{t-1}z_t$ and $\{x_t,z_t\} \in E(G)$. 
This implies that  $z_t \in N_G(x_t)=\{y_1,y_2,...,y_d\}$ and hence ${\bf z} \in \{{\bf y_1},{\bf y_2},...,{\bf y_d}  \}$.
If $i<t$, then we must have
$${\bf x}=x_1x_2...x_{i-1}x_iz_iz_i...z_i  ~~~,~~~ {\bf z}=x_1x_2...x_{i-1}z_ix_ix_i...x_i $$
Thus, $x_t=z_i$ and $x_i\in N_G(z_i)=N_G(x_t)=\{y_1,y_2,...,y_d\}$. 
This implies that $$ {\bf x}=x_1x_2...x_{i-1}x_ix_tx_t...x_t  ~~~,~~~ {\bf z}=x_1x_2...x_{i-1}x_tx_ix_i...x_i  $$
Note  that  $i=\max\{j:~1\leq j\leq t-1,~ x_j\neq x_t\}$.  
Since the vertex ${\bf x}$ is given and its structure is specified, the index $i$ is unique.   
Therefore, we have
$$N_{S(G,t)} (x_1x_2...x_t)\! = \!
\left\{ \begin{array}{ll}  
\!\!\! \{{\bf y_1},{\bf y_2},...,{\bf y_d}, x_1x_2...x_{i-1}x_tx_ix_i...x_i\} &  ~~  x_1x_2...x_t\neq x_tx_t...x_t , ~ x_i \in N_G(x_t) , \\  \!\!\! \{{\bf y_1},{\bf y_2},...,{\bf y_d} \} & ~~ otherwise,  
\end{array} \right.$$
in which $i={\max\{j:~1\leq j\leq t-1,~ x_j\neq x_t\}}$. This completes the proof.
}\end{proof}

Let $G$ be a graph  of order $n\geq 2$ with $E(G)\neq \emptyset$ such that $\delta(G)=\deg_G(x)$, $\Delta(G)=\deg_G(y)$ and $z\in N_G(y)$. 
Lemma \ref{VertexDegree} implies that $\deg_{S(G,t)}(xx...x)=\deg_G(x)=\delta(G)$ and for $t\geq 2$ we have $\deg_{S(G,t)}(zz...zy)=1+\deg_G(y)=1+\Delta(G)$.
Hence, the following  corollary directly follows from  Lemma \ref{VertexDegree}, (also, see \cite{2017Alberto}).

\begin{corollary} \label{MaxDeg}
Let $G$ be a simple   graph of order $n\geq 2$  and  $t\geq 1$ be an integer. Then, 
\begin{itemize}
\item[i)] $\delta(S(G,t))= \delta(G)$. 
\item[ii)] $\Delta(S(\overline{K_n},t)) = \Delta(\overline{K_n})=0$ and $\Delta(S(G,1)) = \Delta(G)$. 
Also, we have $\Delta(S(G,t)) = 1+\Delta(G)$ when $E(G)\neq \emptyset$ and  $t\geq 2$,
\end{itemize}
\end{corollary}

\begin{theorem} \label{DegreeSequence}
Let $G=(V,E)$ be a simple graph of order $n\geq 2$,  $t\geq 1$ be an integer and for each $k$, $\delta(G) \leq k \leq \Delta(G)$, let $V_k=\{v\in V:~\deg_G(v)=k\}$. 
Then, for each $k$, $\delta(S(G,t)) \leq k \leq \Delta(S(G,t))$, the number of vertices of degree $k$ in the generalized Sierpi\'nski graph $S(G,t)$   is 
$$ |V_k| ~  n^{t-1} -  {n^{t-1} -1 \over n-1} ~ \big(k~|V_k| - (k-1)~ |V_{k-1}| \big). $$
\end{theorem}
\begin{proof}{
It is easy to check that the result directly follows for the case $t=1$. Hereafter we assume that $t\geq 2$.
For each integer $s$, $\delta(G) \leq s \leq \Delta(G)$, and for each $x \in V_s$ define 
$$\Omega_{s,x}=\{ v_1v_2...v_t \in V(S(G,t)) :~v_t=x  \}$$
and let $\Omega_s=\cup_{x\in V_s} \Omega_{s,x}$. 
Note that $|\Omega_{s,x}|=n^{t-1}$ and $|\Omega_s|= |V_s|~n^{t-1}$.
By Lemma \ref{VertexDegree}, the degree of each vertex in $\Omega_s$ is  $s$ or $s+1$.
Now we want to determine the number of vertices in $\Omega_s$ which has degree $s$. 
Let $x\in V_s$ and $v_1v_2...v_{t-1}x \in \Omega_s$. 
If $\deg_{S(G,t)} (v_1v_2...v_{t-1}x)=s$, then by Lemma \ref{VertexDegree} we have $v_1v_2...v_{t-1}x=xx...x$ or
there exists $1\leq i \leq t-1$ such that $i={\max\{j:~1\leq j\leq t-1,~ v_j\neq x \}}$ and $v_i \notin N_G(x)$. 
For each $i$, $1\leq i \leq t-1$, define
$$\Gamma_{s,x,i}=\{x_1x_2...x_t \in V(S(G,t)):~ x_i\notin N_G(x)\cup \{x\},~ x_j=x ~~\forall j>i  \}.$$
Note that $|\Gamma_{s,x,i}|=n^{i-1}~(n-s-1)$. 
Thus, the degree of the vertex $v_1v_2...v_{t-1}x \in \Omega_s$ is equal to $s$ if and only if 
$$v_1v_2...v_{t-1}x \in \left(  \cup_{i=1}^{t-1} ~ \Gamma_{s,x,i} \right) \cup \{xx...x\}.$$
Since $x\in V_s$ and
$$\big{|} \left(  \cup_{i=1}^{t-1} ~ \Gamma_{s,x,i} \right) \cup \{xx...x\} \big{|} = 1+ (n-s-1) ~ (1+n+n^2+\cdots +n^{t-2}),$$
the number of vertices in $\Omega_s$ of degree $s$ is 
$$|V_s|~ \big{(} 1+ (n-s-1) ~ (1+n+n^2+\cdots +n^{t-2}) \big{)} = |V_s|~ \big{(} n^{t-1} - s ~ (1+n+\cdots +n^{t-2}) \big{)}.$$
Hence, the number of vertices in $\Omega_s$ of degree $s+1$ is given by
\begin{eqnarray*}
|\Omega_s| - |V_s|~ \big{(} n^{t-1} - s ~ (1+n+\cdots +n^{t-2}) \big{)} &=& |V_s|~n^{t-1} - |V_s|~ \big{(} n^{t-1} - s ~ (1+n+\cdots +n^{t-2}) \big{)} \\
&=&  s~|V_s|~ (1+n+\cdots +n^{t-2}).
\end{eqnarray*}
Therefore, the number of vertices of degree $k$ in the generalized Sierpi\'nski graph $S(G,t)$ is equal to
$$ |V_k|~ \big{(} n^{t-1} - k ~ (1+n+\cdots +n^{t-2}) \big{)}  +  (k-1)~|V_{k-1}|~ (1+n+\cdots +n^{t-2}) $$
Since $V_{_{1+\Delta(G)}}=\emptyset=V_{_{-1}}$ and $1+n+\cdots +n^{t-2} = {n^{t-1}-1 \over n-1}$, the proof is completed.
}\end{proof}

A large number of properties like chemical activity, biological activity, physicochemical properties, thermodynamic properties are determined by the  chemical applications
 of graph theory. These properties can be expressed by certain graph invariants (real numbers related to a graph which is structurally invariant) referred to as topological indices. 
Some of topological indices of Sierpi\'nski networks and generalized Sierpi\'nski graphs are determined, see  \cite{GeneralRandic}, \cite{Farahani} and \cite{Randic2015}.
In \cite{LiZhao} and \cite{M1k}  Li et al. considered the general first  Zagreb index of a graph $G$  as
$$Z_\alpha(G)=\sum_{\{u,v\}\in E(G)} \left((\deg_G(u))^{\alpha -1}+(\deg_G(v))^{\alpha -1}\right)=\sum_{u\in V(G)} (\deg_G(u))^\alpha$$
in which $\alpha$ is a real number. 
Specially, we see that $Z_0(G)=n$, $Z_1(G)=\sum_{k=1}^{\Delta(G)} ~|V_k|~k=2|E(G)|$, $Z_2(G)=M_1(G)$ which is known as  the first Zagreb index and $Z_3(G)=F(G)$ which is known as the forgotten topological index, see \cite{Reviewer1-4}, \cite{StarSequence} and \cite{AAA2} for more details.

\begin{corollary}  \label{GeneralFirstZagreb}
For each simple graph $G$ of order $n\geq 2$ and each integer $\alpha\geq 0$,  the general first Zagreb index of the generalized Sierpi\'nski graph $S(G,t)$, $t\geq 1$,  is given by
$$ Z_\alpha(S(G,t)) = {n^t - n^{t-1} + (n^{t-1}-1)\alpha \over n-1} ~Z_\alpha(G) + {n^{t-1} -1 \over n-1} \sum_{j=1}^{\alpha - 1} {\alpha \choose j-1} ~ Z_j(G).$$
\end{corollary}
\begin{proof}{
By the definition of the general first Zagreb index, binomial expansion formula and  Theorem \ref{DegreeSequence}, we have
\begin{eqnarray*}
Z_\alpha(S(G,t))&=& \sum _{k=1}^{\Delta(G)+1}  \left( |V_k| ~  n^{t-1} -  {n^{t-1} -1 \over n-1} ~ \big(k~|V_k| - (k-1)~ |V_{k-1}| \big)  \right)k^\alpha  \\
&=&  \sum _{k=1}^{\Delta(G)+1}  \left( |V_k| ~  n^{t-1}~k^\alpha -  {n^{t-1} -1 \over n-1} ~ \big(k^{\alpha +1}~|V_k| - k^\alpha (k-1)~ |V_{k-1}| \big)  \right)  \\
&=& n^{t-1} ~ Z_{\alpha}(G) - {n^{t-1} -1 \over n-1} ~ Z_{\alpha +1}(G) + {n^{t-1} -1 \over n-1} \sum _{k=1}^{\Delta(G)+1} 
\left( ~k^\alpha (k-1)~ |V_{k-1}|~ \right) \\
&=& n^{t-1} ~ Z_{\alpha}(G) - {n^{t-1} -1 \over n-1} ~ Z_{\alpha +1}(G) +{n^{t-1} -1 \over n-1} \sum _{\ell =1}^{\Delta(G)} 
\left( ~(\ell +1 )^\alpha \ell ~ |V_\ell |~ \right) \\ 
&=& n^{t-1} ~ Z_{\alpha}(G) - {n^{t-1} -1 \over n-1} ~ Z_{\alpha +1}(G) +{n^{t-1} -1 \over n-1} \sum _{\ell =1}^{\Delta(G)} 
\left( \bigg{(}  \sum_{i=0}^\alpha {\alpha \choose i} \ell^i  \bigg{)} ~ \ell ~ |V_\ell | \right) \\ 
&=& n^{t-1} ~ Z_{\alpha}(G) - {n^{t-1} -1 \over n-1} ~ Z_{\alpha +1}(G) +{n^{t-1} -1 \over n-1} 
\sum_{i=0}^\alpha {\alpha \choose i} ~ Z_{i+1}(G) \\
&=&  {n^t - n^{t-1} + (n^{t-1}-1)\alpha \over n-1} ~Z_\alpha(G) + {n^{t-1} -1 \over n-1} \sum_{j=1}^{\alpha - 1} {\alpha \choose j-1} ~ Z_j(G).
\end{eqnarray*}
}\end{proof}
Note that  by using the generalized form of the binomial theorem, Corollary \ref{GeneralFirstZagreb} can be motivated 
in such a way that  $\alpha$ be a real number  but then the finite sum should be replaced by an infinite series.
By using Corollary \ref{GeneralFirstZagreb}, it can be easily seen that
$$|V(S(G,t))|=Z_0(S(G,t))=n^t,~~|E(S(G,t))|={1\over 2} Z_1(S(G,t))={n^t-1 \over n-1}~\! |E(G)|$$
which coincides with the previously obtained results, see \cite{2017Alberto}. Also,  we have
$$M_1(S(G,t))= Z_2(S(G,t))={n^t+n^{t-1}-2 \over n-1} Z_2(G) + {n^{t-1}-1\over n-1} 2|E(G)|$$
and 
$$F(S(G,t))= Z_3(S(G,t))={n^t+2n^{t-1}-3\over n-1} F(G) + {n^{t-1}-1\over n-1}(2|E(G)|+3M_1(G)).$$
In \cite{Behtoei},  by using the Stirling numbers of the first kind, it is shown that  for each integer $\alpha\geq \Delta(G)$, the general first Zagreb index $Z_\alpha(G)$ can be expressed as a linear combination of $Z_0(G)$, $Z_1(G)$, ..., $Z_{\Delta(G)-1}(G)$. This result using Corollary \ref{GeneralFirstZagreb} implies that 
$Z_\alpha(S(G,t))$ can also be expressed as a linear combination of $Z_0(G)$, $Z_1(G)$, ..., $Z_{\Delta(G)-1}(G)$ for each $\alpha\geq \Delta(G)$. 
\\ \\ 
The authors declare that they have no competing interests.



\end{document}